\documentclass[reqno,11pt]{amsart}

\usepackage{amsmath, amsthm, amssymb, amsfonts, amsthm}
\usepackage{amsxtra, amscd, geometry, graphicx,epic,eepic}
\usepackage[all,cmtip]{xy}
\usepackage{mathrsfs}
\usepackage{hyperref}
\usepackage{graphicx, nicefrac}
\usepackage[dvipsnames]{xcolor}
\usepackage{enumerate}
\usepackage{bbm}

\makeatletter
\newsavebox{\@brx}
\newcommand{\llangle}[1][]{\savebox{\@brx}{\(\m@th{#1\langle}\)}%
  \mathopen{\copy\@brx\kern-0.5\wd\@brx\usebox{\@brx}}}
\newcommand{\rrangle}[1][]{\savebox{\@brx}{\(\m@th{#1\rangle}\)}%
  \mathclose{\copy\@brx\kern-0.5\wd\@brx\usebox{\@brx}}}
\makeatother

\theoremstyle{definition}
\newtheorem{theorem}{Theorem}         
\newtheorem{proposition}[theorem]{Proposition}

\newtheorem{lemma}[theorem]{Lemma}

\newtheorem*{Ex1*}{Example 1}
\newtheorem*{Ex2*}{Example 2}
\newtheorem*{Ex3*}{Example 3}
\newtheorem*{Ex4*}{Example 4}
\newtheorem*{remark*}{Remark}

\newtheorem*{rmk*}{Remark}

\newcommand{\N}{\mathcal{N}}

\newcommand{\FF}{\mathfrak{F}}
\renewcommand{\N}{\mathbb N}
\newcommand{\Z}{\mathbb Z}
\newcommand{\RR}{\mathbb R}

\numberwithin{equation}{section}

\title{A note on the pair correlation of Farey fractions}

\author{Florin P. Boca}
\author{Maria Siskaki}

\address{Department of Mathematics, University of Illinois at Urbana-Champaign,
Urbana, IL 61801}

\address{E-mail: fboca@illinois.edu \quad siskaki2@illinois.edu}

\begin{document}

\date{\today}

\begin{abstract}
The pair correlations of Farey fractions with denominators $q$ satisfying $(q,m)=1$, respectively $q\equiv b \pmod{m}$ with $(b,m)=1$,
are shown to exist and are explicitly computed.
\end{abstract}

\maketitle

\section{Introduction}
The Farey fractions sequence $\FF_Q:=\{ \frac{a}{q} : 0< a\leq q\leq Q, (a,q)=1\}$  arises in several problems in mathematics.
The elements of $\FF_Q$ are well known to be uniformly distributed in $[0,1]$ as $Q\rightarrow\infty$ \cite{Mik}, with
discrepancy exactly $\frac{1}{Q}$ \cite{Dr}.
The distribution of Farey fractions is of major interest, due in part to the connection with the distribution of zeros
of the Riemann zeta function \cite{Fr,La} or of Dirichlet $L$-functions \cite{Hux}.

Although the major problems in the area remain widely open, the spacing statistics of Farey fractions are more accessible.
The gap distribution of $h$-tuples of consecutive gaps between elements of $\FF_Q$ was computed in \cite{Ha} for $h=1$ and in \cite{ABCZ}
for $h\geq 2$. More recently, the correlations of $\FF_Q$, shown to exist and explicitly computed by Zaharescu and the
first author \cite{BZ1}, turned out to play a key role in the study of the moments of eigenvalues of large sieve matrices \cite{BR}.

Motivated by Huxley's work \cite{Hux}, a number of papers investigated various features (such as discrepancy or gap distribution)
of the distribution of Farey fractions with denominators subjected to various constraints \cite{ALVZ1,ALVZ2,BH,BHS,He,Le}.

For every finite set $F\subseteq \RR$ of cardinality $N(F)$ and every interval $I$, define
\begin{equation}\label{eq1.1}
{\mathcal G}_F(I):=\frac{1}{N(F)}\, \# \bigg\{ (x,y)\in F^2 : y\neq x, \, y-x \in \frac{1}{N(F)}\, I+\Z \bigg\}.
\end{equation}
The \emph{pair correlation measure} of an increasing sequence $(F_n)$ of finite subsets of $\RR$ is defined (when it exists) by
$$
{\mathcal G} (I):= \lim\limits_n {\mathcal G}_{F_n} (I) \qquad (\text{$I$ interval}).
$$
If, in addition,
$$
G(\lambda) :=  {\mathcal G} ( [0,\lambda]) =\int_0^\lambda g(x)\, dx,
$$
then $g$ is called the \emph{pair correlation function} of $(F_n)$.

The pair correlation function of $\FF_Q$ was shown in \cite{BZ1} to be given by
\begin{equation}\label{eq1.2}
g(\lambda) =\frac{1}{\zeta(2) \lambda^2} \sum\limits_{1\leq \Delta \leq 2\zeta(2)\lambda} \varphi(\Delta) \log \frac{2\zeta (2) \lambda}{\Delta} ,\qquad \forall \lambda >0.
\end{equation}
This formula was useful in \cite{BR} to recognize the connection between the pair correlation of $\FF_Q$ and the
expression of the main term of the second moment of the large sieve matrix, provided in \cite{Ra}.

The proof of formula \eqref{eq1.2} given in \cite{BZ1} relies essentially on the Poisson summation formula.
The original motivation of this note was to re-prove \eqref{eq1.2} using some different counting arguments
that also provide effective estimates.
Our direct approach turns out to also work well in the case of two important subsets of $\FF_Q$, obtained by imposing
congruence conditions on the denominators:
$$
\begin{aligned}
\FF_Q^{(m)} & :=  \bigg\{\gamma= \frac{a}{q} \in \FF_Q : (q,m)=1\bigg\},\\
\FF_Q^{(m,b)} & :=\bigg\{ \frac{a}{q} \in \FF_Q :  q\equiv b\hspace{-6pt} \pmod{m}\bigg\},
\end{aligned}
$$
where $m\in \N$, $b\in \Z$ and $(b,m)=1$.

Set
$$ N_{Q,m;\alpha,\beta} := \# (\FF_Q^{(m)} \cap (\alpha, \beta]), \quad \text{and } \quad
    N_{Q,m} := N_{Q,m;0,1} .
$$
The following constant will appear several times in this paper:
$$
C_m :=\frac{\varphi(m)}{\zeta (2)m} \prod_{\substack{p|m \\ p \text{ prime}}}\bigg(1-\frac{1}{p^2}\bigg)^{-1}
=\frac{\varphi(m)}{m} \prod_{\substack{p \nmid m \\ p \text{ prime}}}\bigg(1-\frac{1}{p^2}\bigg).
$$
As noticed at the beginning of Section 2, for every $0\leq \alpha <\beta \leq 1$ we have
$$
\begin{aligned}
 N_{Q,m;\alpha,\beta} & = (\beta -\alpha) N_{Q,m} + O_\delta(Q^{1+\delta}) \\ &
 =\frac{(\beta-\alpha)C_m}{2}\, Q^2 +O_\delta(Q^{1+\delta}) ,\quad \forall \delta >0,
 \end{aligned}
$$
which gives an effective estimate for the uniform distribution of $\FF_Q^{(m)}$.

In the first part of Section 2 we prove
\begin{theorem}\label{thm1}
The pair correlation function $g_{(m)}$ of $\FF_Q^{(m)} $ exists and
\begin{equation*}
    g_{(m)} (\lambda)= \frac{\varphi(m)}{m} \cdot \frac{C_m}{\lambda^2} \sum\limits_{1\leq \Delta \leq \frac{2\lambda}{C_m}} \varphi(\Delta) \,
    \frac{ (\Delta,m)}{ \varphi\big((\Delta,m)\big)} \, \log  \frac{2 \lambda}{C_m \Delta}.
\end{equation*}
\end{theorem}
In particular, the support of the function $g_{(m)}$ is the interval $[\frac{1}{2} C_m,\infty)$.

Next, we extend the equality
$
\lim\limits_{\lambda \rightarrow \infty} g_{(1)} (\lambda)=1,
$
due to R. R. Hall and presented in \cite{BZ1}, by proving that
\begin{equation}\label{eq1.3}
\lim\limits_{\lambda\rightarrow \infty} g_{(m)} (\lambda)=1,\quad \forall m\in \N .
\end{equation}

In Section 3 we investigate the pair correlation of $\FF_Q^{(m,b)}$ under the assumption
$(b,m)=1$. The cardinality
of $\FF_Q^{(m,b)}$ is given by
\begin{equation}\label{eq1.4}
N_{Q,(m,b)} = \frac{C_m}{2\varphi(m)} \, Q^2 +O_m (Q\log Q) .
\end{equation}
Furthermore, we have
$$
\# ( \FF_Q^{(m,b)} \cap (\alpha,\beta]) =(\beta-\alpha) N_{Q,(m,b)}  +O_\delta (Q^{1+\delta}), \quad \forall \delta >0,
$$
showing effectively that the elements of $\FF_Q^{(m,b)}$ are uniformly distributed.

We prove
\begin{theorem}\label{thm2}
The pair correlation function of $\FF_Q^{(m,b)} $ exists, is independent of $b$, and is given by
$$
\widetilde{g}_{(m)} (\lambda)=g_{(m)} ( \varphi (m)\lambda) .
$$
 \end{theorem}

Our approach also allows us to prove that the pair correlation function of $\FF_Q^{(m)} \cap I$, respectively
$\FF_Q^{(m,b)} \cap I$, coincides with $g_{(m)}$, respectively $\widetilde{g}_{(m)}$, for every interval $I\subseteq [0,1]$.
It also gives effective asymptotic formulas in $Q$ for the quantities ${\mathcal G}_{\FF_Q} ([0,\lambda])$,
${\mathcal G}_{\FF_Q^{(m)}} ([0,\lambda])$ and ${\mathcal G}_{\FF_Q^{(m,b)}} ([0,\lambda])$.

A related result is contained in \cite[case $n=1$ of Theorem 3.2]{Xi}. However, our result involves the extra coprimality
condition $(a,q)=1$ in the definition of $\FF_Q^{(m,b)}$, which is not included in formulas (3.3.1) and (3.3.2)  of \cite{Xi}.

\begin{figure}[h]
\hspace{-2cm}
\includegraphics[scale=0.52, bb = 80 0 250 250]{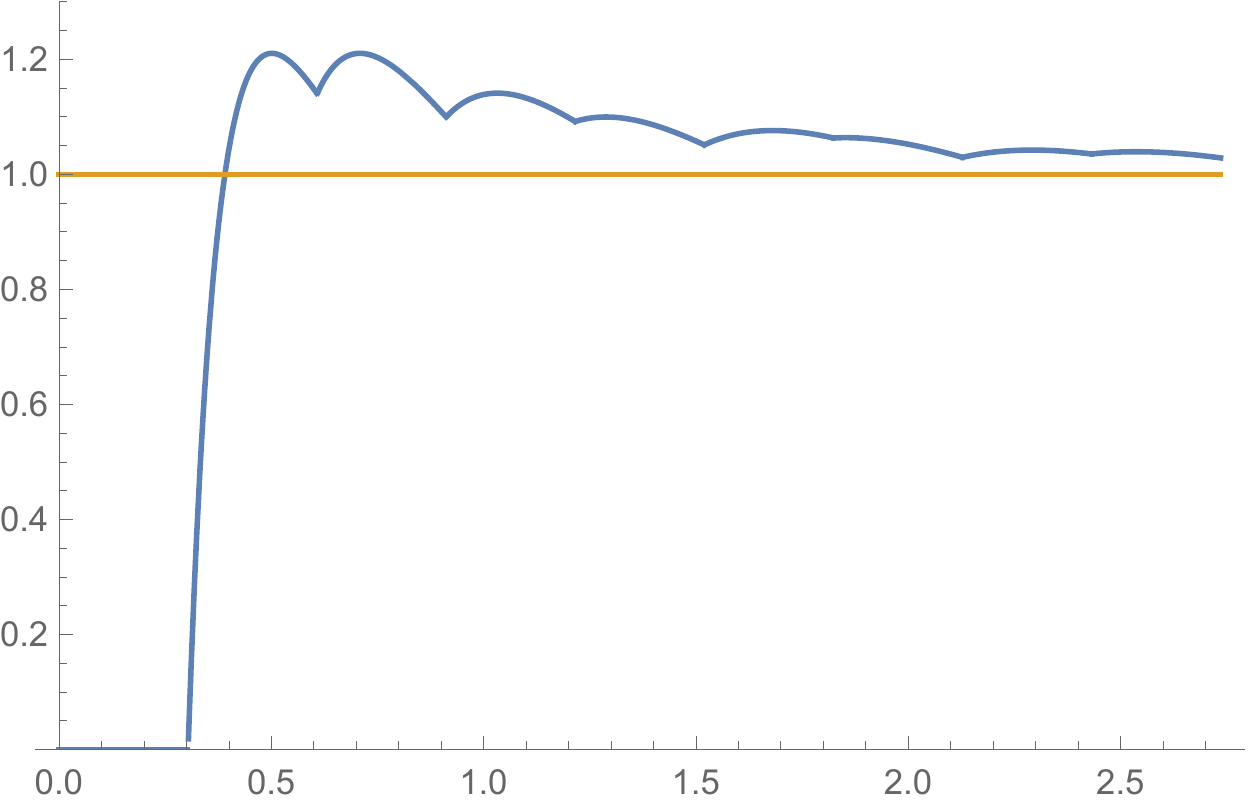}
\hspace{4cm}
\includegraphics[scale=0.52, bb = 80 0 250 250]{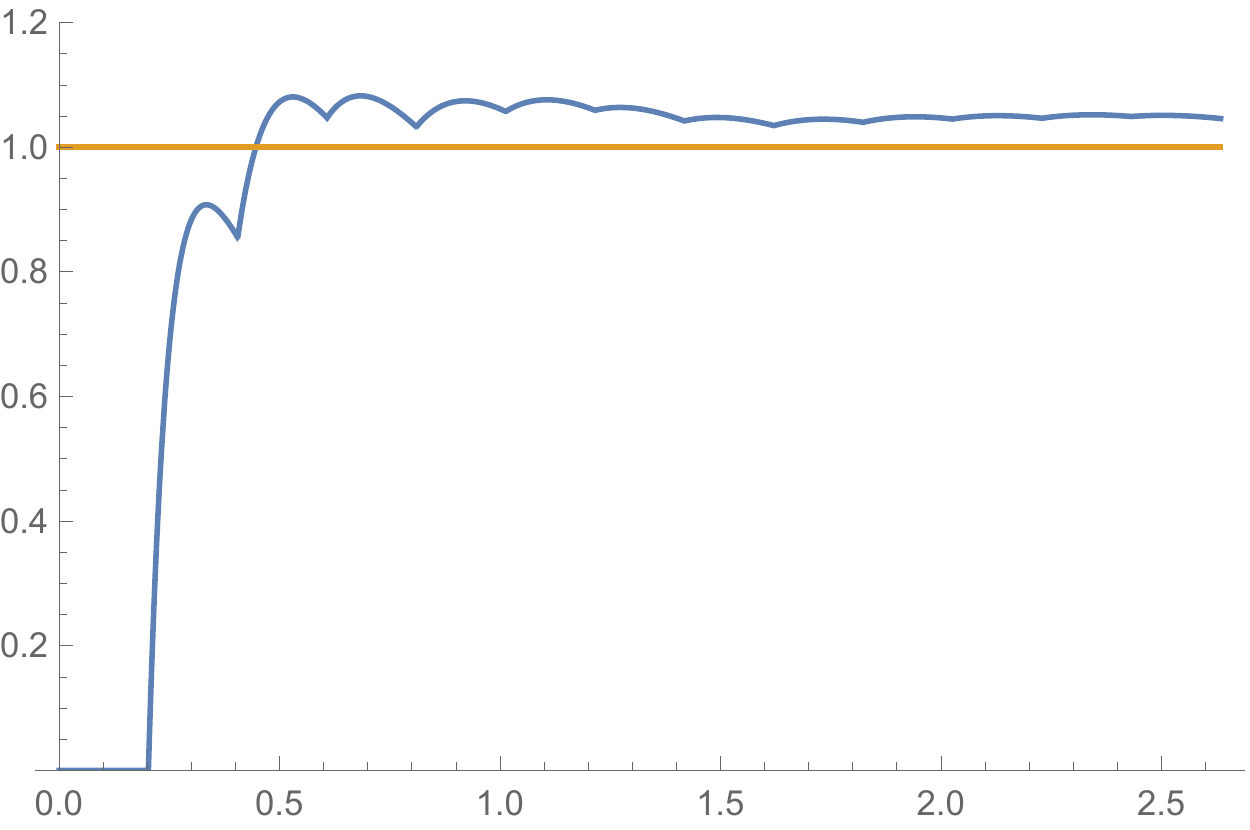}

\includegraphics[scale=0.52, bb = 80 0 250 250]{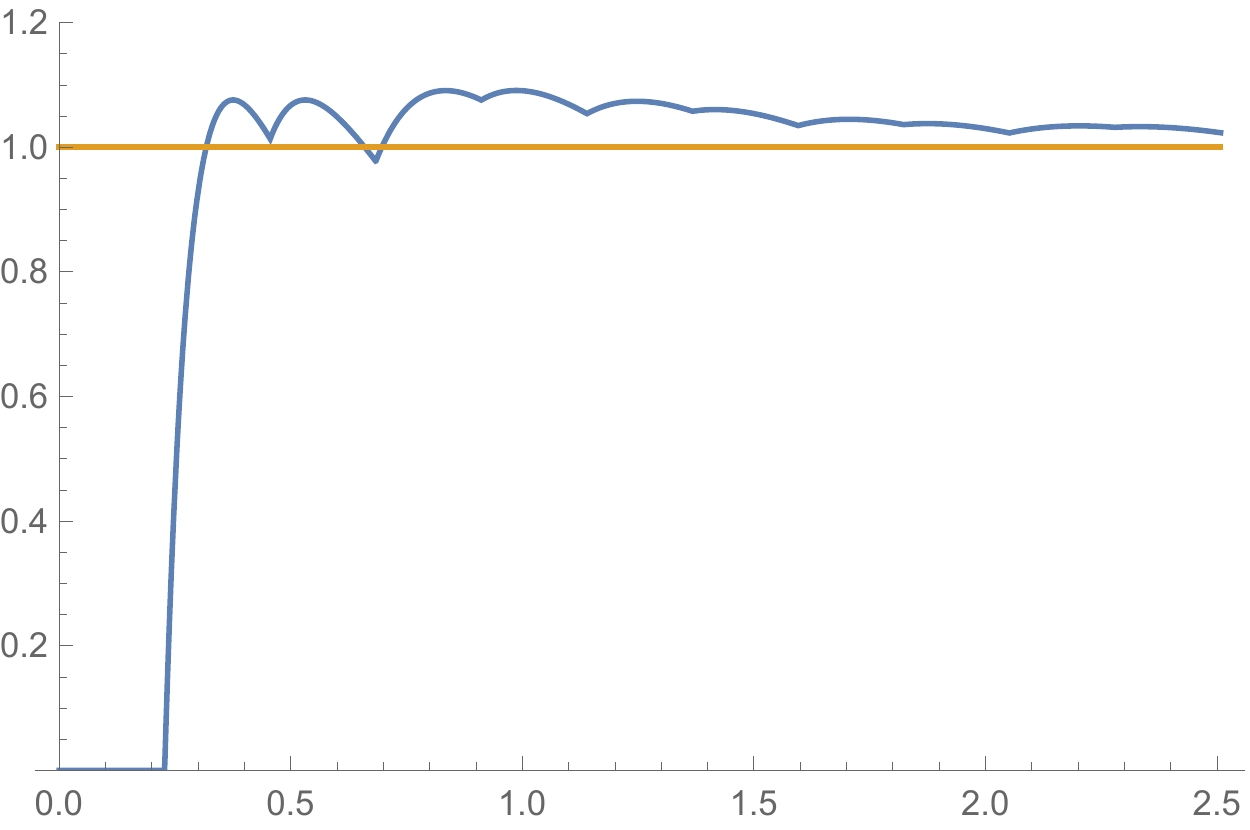}
\caption{The pair correlation functions $g_{(1)}$, $g_{(2)}$ and $g_{(3)}$}\label{Figure}
\end{figure}

\section{The pair correlation of $\FF_Q^{(m)}$}
Lemma 2.1 of \cite{BG} gives
\begin{equation*}
   N_{Q,m} = 1+ \sum\limits_{\substack{k=1 \\ (k,m)=1}}^{Q} \varphi(k) = C_m  \frac{Q^2}{2} \ + \ O(Q \log Q) .
\end{equation*}
When restricting to  $[0, \beta]$, the number of new fractions in the $k$th step is not $\varphi(k)$ anymore, but rather
\begin{equation*}
    \sum\limits_{\substack{n=1 \\ (n,k)=1}}^{\lfloor k\beta \rfloor}1= \frac{\varphi(k)}{k} \lfloor k\beta \rfloor +O_\delta(k^{\delta}),
\end{equation*}
where one can use, for example, \cite[Lemma A.1]{BZ2}. Thus
\begin{align*}
    N_{Q,m;0,\beta} &= \beta N_{Q,m} + O_\delta(Q^{1+\delta}), \\
N_{Q,m;\alpha,\beta} & = N_{Q,m;0,\beta} - N_{Q,m;0,\alpha} = (\beta -\alpha) N_{Q,m} + O_\delta(Q^{1+\delta}).
\end{align*}
Set
\begin{align*}
  H_{Q,m;\beta} (\lambda) & := \#\bigg\{(\gamma,\gamma'): \gamma,\gamma^\prime \in \FF_Q^{(m)}, 0<\gamma'-\gamma \leq \frac{\lambda}{Q^2}, \ \gamma' \leq \beta \bigg\}, \\
  \overline{H}_{Q,m;\alpha} (\lambda) & := \#\bigg\{(\gamma,\gamma'):\gamma,\gamma^\prime  \in \FF_Q^{(m)}, 0<\gamma'-\gamma \leq \frac{\lambda}{Q^2}, \ \gamma \leq \alpha\bigg\}.
  \end{align*}
   If the limit exists, set
 $$
  G_{(m;\alpha,\beta)} (\lambda)  := \lim_{Q \rightarrow \infty} \frac{1}{N_{Q,m;\alpha,\beta}  (Q)}\,\#\left\{(\gamma,\gamma')  :
 \begin{matrix}  \gamma,\gamma^\prime \in \FF_Q^{(m)} \\ 0< \gamma'-\gamma \leq \frac{\lambda}{N_{Q,m;\alpha,\beta}} \end{matrix}\right\}.
$$

Then
\begin{align*}
  H_{Q,m; \beta}(\lambda)& = \sum\limits_{1 \leq \Delta \leq \lambda }  \# S_{Q,m;\beta}(\Delta, \lambda)
  = \sum\limits_{1 \leq \Delta \leq \lambda }  \# \widetilde{S}_{Q,m;\beta}(\Delta, \lambda) ,\\
   \overline{H}_{Q,m; \alpha}(\lambda)& = \sum\limits_{1 \leq \Delta \leq \lambda }  \# \overline{S}_{Q,m;\alpha}(\Delta, \lambda),
\end{align*}
where we used the variables $x=a', v=q', u=q, y=a$ to get
\begin{align*}
   & S_{Q,m;\beta}(\Delta, \lambda)  =  \left\{(u,v,x,y) \in \mathbb{N}^4 : \begin{matrix}
   xu-yv=\Delta, \ \substack{x \leq \beta v, \ v \leq Q\\ y \leq u \leq Q}, \ \substack{ (x,v)=1 \\ (y,u)=1}
   \\ \frac{\Delta}{uv} \leq \frac{\lambda}{Q^2}, \ \substack{(u,m)=1 \\(v,m)=1}\end{matrix} \right\},\\
   & \widetilde{S}_{Q,m;\beta} (\Delta,\lambda) = \left\{ (u,v,x) \in \mathbb{N}^3  : \begin{array}{c}
   \frac{\Delta Q}{\lambda} \leq v \leq Q, \ \frac{\Delta Q^2}{\lambda v}\leq u \leq Q, \ \substack{x \leq \beta v \\ (x,v)=1} \\
    xu \equiv \Delta \mod v ,\
    (\frac{xu-\Delta}{v},u)=1, \ \substack{(u,m)=1 \\(v,m)=1} \end{array}\right\},\\
    & \overline{S}_{Q,m;\alpha}(\Delta, \lambda)  =  \left\{(u,v,x,y) \in \mathbb{N}^4 :
    \begin{matrix} xu-yv=\Delta, \ \substack{x \leq v \leq Q\\ y \leq \alpha u , \ u \leq Q}, \ \substack{ (x,v)=1 \\ (y,u)=1} \\
    \frac{\Delta}{uv} \leq \frac{\lambda}{Q^2}, \ \substack{(u,m)=1 \\(v,m)=1} \end{matrix}\right\}.
\end{align*}
Observe that $xu-yv=\Delta \Leftrightarrow \frac{x}{v}= \frac{y}{u}+ \frac{\Delta}{uv}$ implies
$$
\# S_{Q,m;\beta}(\Delta, \lambda) \leq
\# \overline{S}_{Q,m;\beta}(\Delta, \lambda) \leq \# S_{Q,m;\beta+\frac{\lambda}{Q^2}} (\Delta,\lambda),
$$
so that $\overline{H}_{Q,m;\beta} (\lambda)$
is asymptotically the same as  $H_{Q,m;\beta} (\lambda)$
as $Q\rightarrow \infty$. Thus it suffices to estimate $\# \widetilde{S}_{Q,m;\beta}(\Delta, \lambda)$ as follows:
\begin{align}\label{bar S extimate 1}
  \nonumber  \# \widetilde{S}_{Q,m;\beta}(\Delta, \lambda) & = \sum\limits_{\substack{\frac{\Delta Q}{\lambda} \leq v \leq Q \\ (v,m)=1}}
  \sum\limits_{\substack{\frac{\Delta Q^2}{\lambda v}\leq u \leq Q, \ (u,m)=1 \\ x \leq \beta v, \ (x,v)=1 \\ xu \equiv \Delta \mod v \\ (\frac{xu-\Delta}{v},u)=1}}1 \\
  \nonumber  &= \sum\limits_{\substack{\frac{\Delta Q}{\lambda} \leq v \leq Q \\ (v,m)=1}}
  \sum\limits_{\substack{\frac{\Delta Q^2}{\lambda v}\leq u \leq Q, \ (u,m)=1 \\ x \leq \beta v, \ (x,v)=1 \\ xu \equiv \Delta \mod v}}
  \sum\limits_{\substack{d\mid\frac{xu-\Delta}{v}\\ d\mid u}} \mu(d) \\
    & \stackrel{u=dw}{=} \sum\limits_{\substack{d\mid \Delta \\ (d,m)=1}} \mu(d) \sum\limits_{\substack{\frac{\Delta Q}{\lambda} \leq v \leq Q \\ (v,m)=1}}
    \ \ \sum\limits_{\substack{\frac{\Delta Q^2}{\lambda d v }\leq w \leq \frac{Q}{d},\ (w,m)=1\\ x \leq \beta v, \ (x,v)=1 \\ xw \equiv \frac{\Delta}{d} \mod v}}1.
\end{align}

To estimate the innermost sum on the right hand side in \eqref{bar S extimate 1},
 we need to check that \cite[Proposition A.3]{BZ2} carries over with the additional condition $(q,m)=1$. Set
\begin{equation*}
    \mathcal{N}_{q,h,m}(I_1, I_2) = \{(x,y) \in I_1 \times I_2 : (x,q)=1, \ (y,m)=1, \ xy \equiv h \mod q\}.
\end{equation*}

\begin{proposition}
Assuming $(q,m)=1$, for any intervals $I_1,I_2$ and any integer $h$ we have
\begin{equation*}
    \# \mathcal{N}_{q,h,m}(I_1, I_2) = \frac{\varphi(q)}{q^2}\cdot  \frac{\varphi(m)}{m} \, |I_1||I_2|
     + \ O_{\delta, m} \bigg( q^{1/2+\delta} (h,q)^{1/2}
    \Big( 1+\frac{\lvert I_1\rvert}{q}\Big) \Big( 1+\frac{\lvert I_2 \rvert }{q}\Big)\bigg).
\end{equation*}
\end{proposition}

\begin{proof}
For $x$ such that $(x,q)=1$, let $\overline{x}$ denote the unique inverse of $x \hspace{-4pt}\mod q$. We have that
\begin{align*}
    \# \mathcal{N}_{q,h,m}(I_1, I_2) &= \sum\limits_{\substack{(x,y) \in I_1 \times I_2\\ (x,q)=  (y,m)=1\\ xy \equiv h \mod q}}1
    = \frac{1}{q} \sum\limits_{\substack{(x,y) \in I_1 \times I_2\\ (x,q)=  (y,m)=1}} \sum\limits_{k=0}^{q-1} e\bigg(  \frac{k(y-\overline{x}h)}{q} \bigg).
\end{align*}
We distinguish the cases $k=0$ and $k >0$:
\begin{align*}
    M& := \frac{1}{q} \sum\limits_{\substack{x \in I_1 \\ (x,q)=1}}\sum\limits_{\substack{y \in I_2\\ (y,m)=1}}1,  \\
    E& := \frac{1}{q}\sum\limits_{\substack{y \in I_2\\ (y,m)=1}} \sum\limits_{k=1}^{q-1} e\bigg( \frac{ky}{q}\bigg)
    \sum\limits_{\substack{x \in I_1 \\ (x,q)=1}} e\bigg(-\frac{\overline{x}hk}{q} \bigg).
\end{align*}
For the term $M$, 
two successive applications of \cite[Lemma A1]{BZ2} give
\begin{align*}
    M=& \frac{1}{q}\sum\limits_{\substack{x \in I_1 \\ (x,q)=1}} \bigg( \frac{\varphi(m)}{m}\, |I_2| + O_\delta (m^\delta) \bigg ) \\
   =& \frac{1}{q} \cdot\frac{\varphi(m)}{m}\, |I_2| \sum\limits_{\substack{x \in I_1 \\ (x,q)=1}} 1 +   O_\delta \bigg(m^\delta\frac{|I_1|+1}{q}  \bigg) \\ = &
   \frac{1}{q} \cdot\frac{\varphi(m)}{m}\, |I_2| \bigg(\frac{\varphi(q)}{q}|I_1|+ O_\delta (q^\delta) \bigg) + O_\delta \bigg(m^\delta\frac{|I_1|+1}{q}  \bigg)\\
  =& \frac{\varphi(q)}{q^2}\cdot \frac{\varphi(m)}{m} \, |I_1||I_2| + O_{\delta,m} \bigg(\frac{|I_1|+|I_2|+1}{q^{1-\delta}}\bigg).
\end{align*}

Now, following the notation and approach from \cite{BZ2}, which makes essential use of the Weil-Sali\' e type estimates derived in \cite[(5)]{Est}, we have
\begin{align*}
    E& = \frac{1}{q} \sum\limits_{\substack{y \in I_2\\ (y,m)=1}} \sum\limits_{k=1}^{q-1} e\bigg( \frac{ky}{q}\bigg) S_{I_1}(0,-hk;q)
   \\
    & = \frac{1}{q}\sum\limits_{d\mid m} \mu(d)\sum\limits_{k=1}^{q-1}  S_{I_1}(0,-hk;q)  \sum\limits_{\substack{y \in I_2\\ d\mid y}} e\bigg( \frac{ky}{q}\bigg)
   \\ &  \stackrel{y=d\ell}{=}  \frac{1}{q}\sum\limits_{d\mid m} \mu(d)\sum\limits_{k=1}^{q-1}  S_{I_1}(0,-hk;q)  \sum\limits_{\ell\in \frac{1}{d} I_2} e\bigg( \frac{kd\ell}{q}\bigg).
\end{align*}

We distinguish the cases $q \mid kd$ and $q \nmid kd$. The former cannot occur because
$d\mid m$, $(q,m)=1$ and $q>k$. For the latter, we use \cite[Lemma A2]{BZ2} to estimate $S_{I_1}(0,-hk;q)$. Here, we do not necessarily have $I_1 \subseteq [0,q)$,
so we get the extra factor in the final formula:
\begin{equation}\label{eq2.2}
\begin{aligned}
    S_{I_1}(0,-hk;q) & \ll (hk,q)^{1/2} q^{1/2 + \delta} \bigg( 1+ \frac{|I_1|}{q}\bigg) \\
    & \leq (h,q)^{1/2}(k,q)^{1/2}q^{1/2+\delta} \bigg( 1+ \frac{|I_1|}{q}\bigg).
    \end{aligned}
\end{equation}
Since
$S_{[\ell q,(\ell+1)q)} (0,-hk;q)$ coincides with the Ramanujan sum
$c_q (-hk) =\sum\limits_{d\mid (hk,q)} \mu (\frac{q}{d})d \ll_\delta (hk,q)^{1+\delta}$, the first estimate in
\eqref{eq2.2} can be improved to
$$
S_{I_1} (0,-hk;q) \ll_\delta (hk,q)^{1/2} q^{1/2+\delta} +(hk,q)^{1+\delta} \frac{| I_1 |}{q}.
$$

Combine \eqref{eq2.2} with the geometric sum and the inequality $|\sin \pi x| \geq 2 \|x\|$ to get
\begin{align*}
    E & =  \frac{1}{q}\sum\limits_{d\mid m} \mu(d)\sum\limits_{\substack{k=1 \\ q\nmid kd}}^{q-1}  S_{I_1}(0,-hk;q)
    \sum\limits_{\ell\in \frac{1}{d} I_2} e\bigg( \frac{kd\ell}{q}\bigg) \\
    & \ll \frac{(h,q)^{1/2}q^{1/2+\delta}}{q} \bigg( 1+ \frac{|I_1|}{q}\bigg) \sum\limits_{d \mid m}
    \sum\limits_{\substack{k=1\\ q \nmid kd}}^{q-1} \frac{(k,q)^{1/2}}{\big\| \frac{kd}{q}\big\|}\\
    & \leq \frac{(h,q)^{1/2}q^{1/2+\delta}}{q}  \bigg( 1+ \frac{|I_1|}{q}\bigg)
    \sum\limits_{d \mid m} \sum\limits_{\substack{k=1\\ q \nmid kd}}^{q-1} \frac{(kd,q)^{1/2}}{\big\| \frac{kd}{q}\big\|}.
\end{align*}

Since
$$\{ kd : 1\leq k< q, q\nmid kd\} \subseteq \{ n= cq+r : 1\leq r <q, 0 \leq c <d\},$$
we further get
\begin{align*}
    E & \ll \frac{(h,q)^{1/2+\delta}q^{1/2}}{q} \bigg( 1+ \frac{|I_1|}{q}\bigg) \sum\limits_{d \mid m}
    \sum\limits_{c=0}^{d-1}\sum\limits_{r=1}^{q-1} \frac{(cq+r,q)^{1/2}}{\big\|c +\frac{r}{q} \big\|}\\
    &= \frac{(h,q)^{1/2}q^{1/2+\delta}}{q} \bigg( 1+ \frac{|I_1|}{q}\bigg) \sum\limits_{d \mid m} d \sum\limits_{r=1}^{q-1} \frac{(r,q)^{1/2}}{\big\|\frac{r}{q} \big\|} \\
   & \stackrel{\substack{\ell = (r,q) \\ r = \ell s}}{\leq} \frac{(h,q)^{1/2}q^{1/2+\delta}}{q}  \bigg( 1+ \frac{|I_1|}{q}\bigg)\sum\limits_{d \mid m} d \sum\limits_{\ell \mid q}
   \sum\limits_{s \leq \frac{q}{2\ell}} \frac{2\ell^{1/2}}{\frac{\ell s }{q}} \\
   & \ll_m (h,q)^{1/2} q^{1/2+3\delta} \bigg( 1+ \frac{|I_1|}{q}\bigg) .
\end{align*}

\end{proof}

Thus,
under the correspondence $ v \leftrightarrow q, x \leftrightarrow x, w \leftrightarrow y, \frac{\Delta}{d} \leftrightarrow h$, relation (\ref{bar S extimate 1}) becomes
\begin{align*}
  &   \# \widetilde{S}_{Q,m; \beta}(\Delta, \lambda) = \sum\limits_{\substack{d\mid\Delta \\ (d,m)=1}} \mu(d)
  \sum\limits_{\substack{\frac{\Delta Q}{\lambda} \leq v \leq Q \\ (v,m)=1}}
     \bigg(\mathcal{N}_{v,\frac{\Delta}{d},m} \Big([0, \beta v], \big[0,\tfrac{Q}{d}\big]\Big)
      -\mathcal{N}_{v,\frac{\Delta}{d},m}\Big([0, \beta v],\big[0, \tfrac{\Delta Q^2}{\lambda d v}\big]\Big) \bigg) \\
    & \qquad = \sum\limits_{\substack{d\mid\Delta \\ (d,m)=1}} \mu(d) \sum\limits_{\substack{\frac{\Delta Q}{\lambda} \leq v \leq Q \\ (v,m)=1}}
     \Bigg( \frac{\varphi(v)}{v^2}\cdot \frac{\varphi(m)}{m} \, \beta v  \bigg( \frac{Q}{d}-  \frac{\Delta Q^2}{\lambda v d}\bigg)
      + O_{\delta, m,\Delta}\bigg(v^{1/2+\delta}\Big(\frac{Q}{v}+\frac{Q^2}{\lambda v^2}\Big)\bigg) \Bigg) \\
     & \qquad = \beta Q \, \frac{\varphi(m)}{m} \sum\limits_{\substack{d\mid\Delta \\ (d,m)=1}} \frac{\mu(d)}{d} \sum\limits_{\substack{\frac{\Delta Q}{\lambda} \leq v \leq Q \\ (v,m)=1}}
     \frac{\varphi(v)}{v} \bigg(1-\frac{\Delta Q}{\lambda v}\bigg) \ + \ O_{\delta,m, \Delta}(Q^{3/2+\delta}).
\end{align*}
Note that there is no dependence of $\lambda$ in the error term because the inequality $\gamma^\prime -\gamma \geq \frac{1}{Q^2}$,
$\forall \gamma,\gamma^\prime \in \FF_Q$, $\gamma < \gamma^\prime$ allows for $\lambda \geq 1$.

A simple calculation shows that the function
$$
K_m (n)= \sum\limits_{\substack{d\mid n \\ (d,m)=1}} \frac{\mu(d)}{d}
$$
is multiplicative for every $m$, and that at prime powers we have
\begin{equation*}
    K_m (p^\ell)= \sum\limits_{\substack{d\mid p^\ell \\ (d,m)=1}} \frac{\mu(d)}{d}
    = \begin{cases} 1, & \text{ if } p\mid m \\ 1-\frac{1}{p} = \frac{\varphi(p^\ell)}{p^\ell}, & \text{ if } p\nmid m \end{cases}.
\end{equation*}
Therefore
\begin{equation*}
    K_m (\Delta) = \prod_{\substack{p\mid\Delta \\ p \nmid m}} \bigg(1-\frac{1}{p}\bigg)
    = \frac{\prod\limits_{\substack{p\mid\Delta}} \big(1-\frac{1}{p}\big)}{\prod\limits_{\substack{p\mid\Delta \\ p  |m}} \big(1-\frac{1}{p}\big)}
    = \frac{\prod\limits_{\substack{p\mid\Delta}} \big(1-\frac{1}{p}\big)}{\prod\limits_{\substack{p\mid (\Delta,m)}} \big(1-\frac{1}{p}\big)}
    = \frac{\frac{\varphi(\Delta)}{\Delta}}{\frac{\varphi\big((\Delta,m)\big)}{(\Delta,m)}},
\end{equation*}
and consequently
\begin{equation*}
    \# \widetilde{S}_{Q,m;\beta}(\Delta, \lambda) = \beta Q \, \frac{\varphi(m)}{m}\cdot \frac{\varphi(\Delta)}{\Delta}
    \cdot\frac{ (\Delta,m)}{ \varphi\big((\Delta,m)\big)}
    \sum\limits_{\substack{\frac{\Delta Q}{\lambda} \leq v \leq Q \\ (v,m)=1}}  \frac{\varphi(v)}{v} \bigg(1-\frac{\Delta Q}{\lambda v}\bigg)
    + \ O_{\delta,m, \Delta}(Q^{3/2+\delta}).
\end{equation*}
Using \cite[Lemma 2.1]{BG} twice, we get
\begin{equation*}
    \# \widetilde{S}_{Q,m;\beta}(\Delta, \lambda) =\beta  C_m Q^2 \, \frac{\varphi(m)}{m} \cdot\frac{\varphi(\Delta)}{\Delta}
    \cdot\frac{ (\Delta,m)}{ \varphi\big((\Delta,m)\big)}
    \bigg( 1-\frac{\Delta}{\lambda} - \frac{\Delta}{\lambda} \log \frac{\lambda}{\Delta} \bigg)  + \ O_{\delta,m, \Delta}(Q^{3/2+\delta}).
\end{equation*}

Thus, for every $K\geq 1$ we have, uniformly in $\lambda \in [1,K]$,
\begin{equation*}
    H_{Q,m;\beta}(\lambda) = \beta C_m  Q^2 \ \frac{\varphi(m)}{m} \sum\limits_{1 \leq \Delta \leq \lambda} \frac{\varphi(\Delta)}{\Delta}
    \cdot\frac{ (\Delta,m)}{ \varphi\big((\Delta,m)\big)}
    \bigg( 1-\frac{\Delta}{\lambda} - \frac{\Delta}{\lambda} \log \frac{\lambda}{\Delta} \bigg) + \ O_{\delta,m, K}(Q^{3/2+\delta}).
\end{equation*}
This leads in turn to
\begin{equation}\label{eq2.3}
    G_{(m;\alpha, \beta)}(\lambda)=   2 \, \frac{\varphi(m)}{m}
 \sum\limits_{1\leq \Delta \leq \frac{2\lambda}{C_m}} \frac{\varphi(\Delta)}{\Delta}  \bigg( 1-\frac{\Delta  C_m}{2\lambda} - \frac{\Delta C_m}{2\lambda}
 \log \frac{2\lambda}{\Delta C_m} \bigg).
 \end{equation}
Finally, Theorem \ref{thm1} follows by differentiating in \eqref{eq2.3}.

For $m=1$, $\alpha=0$, $\beta=1$, we retrieve Theorem 1 in \cite{BZ1}.
For $m=2$ we get
\begin{equation*}
    g_{(2)}(\lambda) = \frac{1}{3\zeta(2)\lambda^2} \sum\limits_{1\leq \Delta \leq 3\zeta(2)\lambda} \varphi(\Delta) (\Delta,2) \log \frac{ 3\zeta(2)\lambda}{\Delta}.
\end{equation*}

In the remaining part of this section we prove equality \eqref{eq1.3}. Consider the Dirichlet series
$$D_m (s):=\sum\limits_{\Delta=1}^\infty \frac{K_m(\Delta)}{\Delta^{s-1}} =\prod\limits_p \bigg( 1+\frac{K_m(p)}{p^{s-1}} +
\frac{K_m(p^2)}{p^{2s-2}} +\cdots \bigg) \quad \text{if } \operatorname{Re} s >2.
$$
Take $\displaystyle \zeta_m (s):=\prod\limits_{p\nmid m} \bigg( 1-\frac{1}{p^s}\bigg)^{-1}$, $\operatorname{Re} s >1$. We have
$$
\frac{\zeta_m(s-1)}{\zeta_m(s)} =\prod\limits_{p\nmid m} \frac{1-\frac{1}{p^s}}{1-\frac{1}{p^{s-1}}} =\prod\limits_{p\nmid m} \frac{p^s-1}{p^s-p}
$$
and
$$
\sum\limits_{\ell=0}^\infty \frac{K_m (p^\ell)}{p^{\ell(s-1)}}
=\begin{cases}
1+\big( 1-\frac{1}{p}\big)\frac{1}{p^{s-1}} \cdot \frac{1}{1-\frac{1}{p^{s-1}}} =1+\frac{p-1}{p^s-p} =\frac{p^s-1}{p^s-p}
& \text{if } p\nmid m \\
\big( 1-\frac{1}{p^{s-1}}\big)^{-1} & \text{if } p\mid m
\end{cases} ,
$$
leading to
$$
\begin{aligned}
D_m (s)& = \frac{\zeta_m (s-1)}{\zeta_m(s)} \prod\limits_{p\mid m} \bigg( 1-\frac{1}{p^{s-1}}\bigg)^{-1} =\frac{\zeta (s-1)}{\zeta(s)}
\prod\limits_{p\nmid m} \frac{p^s-p}{p^s-1} \prod\limits_{p\mid m} \frac{p^{s-1}}{p^{s-1}-1} \\
 & =\frac{\zeta(s-1)}{\zeta(s)} \, c_m(s),\quad \text{where } c_m(s):=\prod\limits_{p\mid m} \frac{1}{1-\frac{1}{p^s}} .
\end{aligned}
$$
Next we follow closely the final part of \cite{BZ1}. By Perron's formula \cite[(4.14)]{BZ1} we infer
$$
\begin{aligned}
\sum\limits_{1\leq \Delta \leq x} \Delta K_m (\Delta) \log  \frac{x}{\Delta} & = \frac{1}{2\pi i}
\int\limits_{\sigma_0-i\infty}^{\sigma_0+i\infty} D_m (s) \, \frac{x^s}{s^2}\, ds \\ & =
\frac{1}{2\pi i} \int\limits_{\sigma_0-i\infty}^{\sigma_0+i\infty} \frac{\zeta (s-1)}{\zeta(s)} \, c_m(s) \, \frac{x^s}{s^2}\, ds
\qquad (\sigma_0 >2).
\end{aligned}
$$
Moving the contour at $\operatorname{Re} s =1$ and employing the notation from \cite{BZ1} we get
$$
\begin{aligned}
\frac{\zeta (it)}{\zeta (1+it)} \prod\limits_{p\mid m} \bigg( 1-\frac{1}{p^{it}} \bigg) & =
\chi(it) \, \frac{\zeta (1-it)}{\zeta (1+it)} \prod\limits_{p\mid m} \bigg( 1-\frac{1}{p^{it}} \bigg) ,\\
 \zeta (1-it) & =\overline{\zeta (1+it)} ,   \\
\underset{s=2}{\operatorname{Res}} \,\frac{\zeta (s-1)}{\zeta (s)} \, c_m(s)\, \frac{x^s}{s^2} & =\lim\limits_{s\rightarrow 2}
\frac{(s-2)\zeta (s-1)}{\zeta (s)} \, c_m(s)\, \frac{x^s}{4} \\
& =\frac{c_m(2)}{\zeta (2)} \cdot \frac{x^2}{4} \quad
\text{($s=2$ is a simple pole)}.
\end{aligned}
$$
Estimating the error as in \cite{BZ1} we find
$$
\begin{aligned}
\sum\limits_{1\leq \Delta \leq x} \Delta K_m(\Delta) \log  \frac{x}{\Delta}&  =
\underset{s=2}{\operatorname{Res}} \,\frac{\zeta (s-1)}{\zeta (s)} \, c_m(s)\, \frac{x^s}{s^2} +O_m(x) \\ &
=\frac{c_m(2)}{\zeta (2)} \cdot \frac{x^2}{4} +O_m(x).
\end{aligned}
$$
Setting $\mu:=\frac{2\lambda}{C_m}$ we get $\lambda=\frac{C_m\mu}{2}$ and (as $\lambda \rightarrow\infty$)
$$
\begin{aligned}
g_{(m)} (\lambda) & =g_{(m)} (\tfrac{1}{2} C_m \mu) =\frac{\varphi(m)}{m} \, C_m \, \frac{4}{C_m^2\mu^2} \sum\limits_{1\leq \Delta \leq \mu}
\Delta K_m(\Delta) \log  \frac{\mu}{\Delta}  \\
& = \frac{4\varphi(m)}{mC_m \mu^2} \left( \prod\limits_{p\mid m} \bigg( 1-\frac{1}{p^2}\bigg)^{-1} \frac{\mu^2}{4\zeta(2)} +O_m (\mu)\right) \\ & =
\frac{1}{\zeta(2)} \prod\limits_{p\nmid m} \bigg( 1-\frac{1}{p^2}\bigg)^{-1} \prod\limits_{p\mid m} \bigg( 1-\frac{1}{p^2}\bigg)^{-1} +O_m(\mu^{-1})\\
&  =1+O_m (\lambda^{-1}).
\end{aligned}
$$

\section{The pair correlation of $\FF_Q^{(m,b)}$}
We will employ the following estimate (\cite[Lemma 3.3]{B}):

\begin{lemma}\label{L1}
Assuming $(b,m)=1$ and $V\in C^1 [0,Q]$, we have
$$
\sum\limits_{\substack{q=1 \\ q\equiv b\hspace{-6pt}\pmod{m}}}^Q
\frac{\varphi(q)}{q}\, V(q) =\frac{C_m}{\varphi(m)} \int_0^Q V+O\Big( (\| V\|_\infty +T_0^N V)\log Q\Big).
$$
\end{lemma}
In particular this gives \eqref{eq1.4}.

Given $\lambda >0$, we are interested in estimating the following three quantities as $Q\rightarrow \infty$:
$$
\begin{aligned}
S_{Q;m,b,\Delta} (\lambda) & :=\# \left\{  (\gamma,\gamma^\prime) : \begin{matrix} \gamma,\gamma^\prime \in \FF_Q^{(m,b)}, \gamma^\prime -\gamma \leq \frac{\lambda}{Q^2} \\
\gamma =\frac{a}{q} < \gamma^\prime =\frac{a^\prime}{q^\prime} ,\
a^\prime q-aq^\prime =\Delta \end{matrix} \right\}, \\
H_{Q;m,b}(\lambda ) & :=\# \bigg\{ (\gamma,\gamma^\prime) : \gamma,\gamma^\prime \in \FF_Q^{(m,b)}, 0< \gamma^\prime -\gamma \leq \frac{\lambda}{Q^2}\bigg\}
=\sum\limits_{1\leq \Delta \leq \lambda} S_{Q;m,b,\Delta}(\lambda)  ,\\
G_{Q;m,b}(\lambda) & := \frac{1}{N_{Q,(m,b)}} \,\# \bigg\{ (\gamma,\gamma^\prime) : \gamma,\gamma^\prime \in \FF_Q^{(m,b)}, 0< \gamma^\prime -\gamma \leq \frac{\lambda}{N_{Q,(m,b)}} \bigg\} .
\end{aligned}
$$
As in the previous section we can write
\begin{equation}\label{eq31}
S_{Q;m,b,\Delta} (\lambda) =\sum\limits_{d\mid \Delta} \mu (d)
\sum\limits_{\substack{\frac{\Delta Q}{\lambda} \leq v\leq Q \\
v\equiv b \hspace{-6pt} \pmod{m}}}
\sum\limits_{\substack{\frac{\Delta Q^2}{\lambda v} \leq u \leq Q \\
u\equiv b \hspace{-6pt} \pmod{m},\, d\mid u \\
x\leq v,\ (x,v)=1 \\
xu \equiv \Delta \hspace{-6pt} \pmod{v}}} 1.
\end{equation}
We write $u=dw$ and observe that the assumption $(b,m)=1$ implies $(d,m)=(v,m)=1$.
Hence we get
\begin{equation}\label{eq32}
S_{Q;m,b,\Delta} (\lambda) =\sum\limits_{\substack{d\mid \Delta \\ (d,m)=1}} \mu (d)
\sum\limits_{\substack{\frac{\Delta Q}{\lambda} \leq v\leq Q \\
v\equiv b \hspace{-6pt} \pmod{m}}} T_{Q;m,b,\Delta}(v,\lambda) ,
\end{equation}
where
\begin{equation}\label{eq33}
T_{Q;m,b,\Delta}(v,\lambda) :=
\sum\limits_{\substack{\frac{\Delta Q^2}{\lambda dv} \leq w \leq \frac{Q}{d} \\
dw \equiv b \hspace{-0.6pt} \pmod{m} \\
x\leq v, \ (x,v)=1 \\
xw \equiv \frac{\Delta}{d}\hspace{-6pt} \pmod{v}}} 1 .
\end{equation}

We write
\begin{equation}\label{eq34}
\begin{aligned}
T_{Q;m,b,\Delta}(v,\lambda) & =
\sum\limits_{\substack{x\in [0,v],\ (x,v)=1 \\ y\in \big[\frac{\Delta Q^2}{\lambda dv} ,\frac{Q}{d}\big] ,\ y\equiv \overline{\overline{d}} b  \hspace{-6pt} \pmod{m}}}
\frac{1}{v} \sum\limits_{k\hspace{-6pt} \pmod{v}} e\bigg( \frac{k(y-\frac{\Delta}{d} \, \overline{x})}{v}\bigg) \\ & =
M_{Q;m,b,\Delta}(v,\lambda) + E_{Q;m,b,\Delta}(v,\lambda),
\end{aligned}
\end{equation}
where $\overline{\overline{d}}$ is the multiplicative inverse of $d\hspace{-4pt} \mod{m}$ and $\overline{x}$ the multiplicative inverse of $x\hspace{-4pt} \mod{v}$,
\begin{equation}\label{eq35}
\begin{aligned}
M_{Q;m,b,\Delta}(v,\lambda) & := \frac{1}{v} \sum\limits_{\substack{x\in [0,v] \\ (x,v)=1}} 1
\sum\limits_{\substack{y\in \big[\frac{\Delta Q^2}{\lambda dv} ,\frac{Q}{d}\big] \\ y\equiv \overline{\overline{d}} b  \hspace{-6pt} \pmod{m}}} 1
\\ & =\frac{1}{v} \bigg( \frac{\varphi(v)}{v} \, v +O_\delta (v^\delta)\bigg) \bigg( \frac{1}{m} \Big( \frac{Q}{d}-\frac{\Delta Q^2}{\lambda dv}\Big) +O(1)\bigg),
\end{aligned}
\end{equation}
\begin{equation}\label{eq36}
E_{Q;m,b,\Delta}(v,\lambda)  := \frac{1}{v} \sum\limits_{k=1}^{v-1} S_{[1,v]} \bigg( 0,-\frac{\Delta}{v} \, k;v\bigg)
\sum\limits_{\substack{y\in \big[\frac{\Delta Q^2}{\lambda dv} ,\frac{Q}{d}\big] \\ y\equiv \overline{\overline{d}} b  \hspace{-6pt} \pmod{m}}}
e \bigg( \frac{ky}{v}\bigg)  .
\end{equation}
We employ the bound\footnote{Here $S_{[1,v]}(0,\ell;v)$ coincides with the Ramanujan sum
$c_v(\ell) \ll_\delta (\ell,v)^{1+\delta}$. On replacing
$v$ by $\beta v$ with $\beta \in (0,1)$ the inequality \eqref{eq37} follows from \eqref{eq2.2}.}
\begin{equation}\label{eq37}
S_{[1,v]} \bigg( 0, -\frac{\Delta}{d}\, k ;v \bigg) \ll_\delta \bigg( \frac{\Delta}{d}\, k,v\bigg)^{1/2} v^{1/2+\delta}
\leq \frac{\Delta}{d} \, (k,v)^{1/2} v^{1/2+\delta},
\end{equation}
$(v,m)=1$, and the geometric series with ratio $e(\frac{km}{v})$ to gather
$$
E_{Q;m,b,\Delta}(v,\lambda) \ll_{\delta,\Delta} v^{-1/2+\delta} \sum\limits_{k=1}^{v-1} (k,v)^{1/2}
\Bigg| \sum\limits_{\substack{y\in \big[\frac{\Delta Q^2}{\lambda dv} ,\frac{Q}{d}\big] \\ y\equiv \overline{\overline{d}} b  \hspace{-6pt} \pmod{m}}}
e \bigg( \frac{ky}{v}\bigg) \Bigg| \ll v^{-1/2+\delta} \sum\limits_{k=1}^{v-1} \frac{(k,v)^{1/2}}{\big\| \frac{km}{v}\big\|} .
$$
Using $\{ km : 1\leq k\leq v-1\} \subseteq \{ \ell : 1\leq \ell \leq mv ,v\nmid \ell\}$ and
$(k,v)^{1/2} \leq (km,v)^{1/2}$
this yields
\begin{equation}\label{eq38}
\begin{aligned}
E_{Q;m,b,\Delta}(v,\lambda) & \ll_{\delta,\Delta} v^{-1/2+\delta} \sum\limits_{\substack{\ell=1 \\ v\nmid \ell}}^{mv}
\frac{(\ell,v)^{1/2}}{\big\| \frac{\ell}{v}\big\|} = mv^{-1/2+\delta}  \sum\limits_{\ell=1}^{v-1} \frac{(\ell,v)^{1/2}}{\big\| \frac{\ell}{v}\big\|}
\\ & \leq 2 m v^{-1/2+\delta}  \sum\limits_{1\leq \ell <\frac{v}{2}} \frac{(\ell,v)^{1/2}}{ \frac{\ell}{v}}
\ll_{\delta,\Delta,m} v^{1/2+3\delta} .
\end{aligned}
\end{equation}
From \eqref{eq34}, \eqref{eq35} and \eqref{eq38} we now infer
$$
S_{Q;m,b,\Delta} (\lambda) =\frac{Q}{m} \sum\limits_{\substack{d\mid \Delta \\ (d,m)=1}} \frac{\mu (d)}{d}
\Bigg( \sum\limits_{\substack{\frac{\Delta Q}{\lambda} \leq v\leq Q \\ v \equiv b\hspace{-6pt} \pmod{m}}} \frac{\varphi (v)}{v}
-\frac{\Delta Q}{\lambda} \sum\limits_{\substack{\frac{\Delta Q}{\lambda} \leq v\leq Q \\ v \equiv b\hspace{-6pt} \pmod{m}}} \frac{\varphi (v)}{v^2} \Bigg)
+O_{\delta,m,\Delta} (Q^{3/2+\delta}) .
$$

Next, an application of Lemma \ref{L1} and the equality
$$
\sum\limits_{\substack{d\mid \Delta \\ (d,m)=1}} \frac{\mu(d)}{d}=\frac{\varphi(\Delta)}{\Delta}
\cdot \frac{(\Delta,m)}{\varphi((\Delta,m))}$$
lead to
\begin{equation}\label{eq39}
S_{Q;m,b,\Delta}(\lambda)  =\frac{C_m Q^2}{m \varphi (m)} \cdot \frac{\varphi(\Delta)}{\Delta} \cdot \frac{(\Delta,m)}{\varphi\big((\Delta,m)\big)}
\bigg( 1-\frac{\Delta}{\lambda}-\frac{\Delta}{\lambda} \, \log \frac{\lambda}{\Delta} \bigg)
 +O_{\delta,m,\Delta} (Q^{3/2+\delta}) ,
\end{equation}
and so we get, for every $K\geq 1$ and uniformly in $\lambda \in [1,K]$,
\begin{equation}\label{eq40}
H_{Q;m,b}(\lambda)  =\frac{C_m Q^2}{m\varphi(m)} \sum\limits_{1\leq \Delta <\lambda} \hspace{-3pt} \varphi(\Delta) \,\frac{(\Delta,m)}{\varphi \big((\Delta,m)\big)}
\bigg( \frac{1}{\Delta}-\frac{1}{\lambda}-\frac{1}{\lambda}\, \log  \frac{\lambda}{\Delta} \bigg)   +O_{\delta,m,K} (Q^{3/2+\delta}).
\end{equation}
Estimates \eqref{eq40}, \eqref{eq1.4}, and the definitions of $G_{Q;m,b}$ and $H_{Q;m,b}$ now yield
\begin{equation}\label{eq41}
G_{Q;m,b}(\lambda) =\frac{1}{N_{Q,(m,b)}}\, H_{Q;m,b} \bigg( \frac{Q^2}{N_{Q;(m,b)}}\, \lambda \bigg)
=G_{(m)}(\lambda)+O_{\delta,m} (Q^{-1/2+\delta}),
\end{equation}
where
$$
G_{(m)} (\lambda) =\frac{2}{m} \sum\limits_{1\leq \Delta \leq K_m \lambda} \varphi(\Delta)\, \frac{(\Delta,m)}{\varphi \big((\Delta,m)\big)}
 \bigg( \frac{1}{\Delta} -\frac{1}{K_m \lambda} -\frac{1}{K_m \lambda}\, \log  \frac{K_m \lambda}{\Delta}\bigg),
$$
with $K_m:=\frac{2\varphi(m)}{C_m}$.

This shows that the pair correlation function of $\FF_Q^{(m,b)}$ exists and is given by
\begin{equation}\label{eq42}
\widetilde{g}_{(m)} (\lambda)  =G_{(m)}^\prime (\lambda)  =\frac{C_m}{m\varphi (m)} \cdot \frac{1}{\lambda^2}
 \sum\limits_{1\leq \Delta \leq K_m\lambda} \varphi(\Delta)\,  \frac{(\Delta,m)}{\varphi \big((\Delta,m)\big)}
 \, \log \frac{K_m \lambda}{\Delta} .
 \end{equation}
Comparing \eqref{eq42} with the formula for $g_{(m)}$ given in Theorem~\ref{thm1} we find that
$$
\widetilde{g}_{(m)} (\lambda)=g_{(m)} ( \varphi (m)\lambda) .
$$

\end{document}